\documentclass[10pt,twoside,a4paper]{article}
\usepackage{amsfonts,amssymb,amsmath,amscd,latexsym,makeidx,theorem}
\author{Luca Martinazzi\thanks{This work was supported by the Swiss National Fond Grant no. PBEZP2-129520.}}
\title{Quantization for the prescribed $Q$-curvature equation on open domains}
\date{March 21, 2010}

\newtheorem{trm}{Theorem}
\newtheorem{prop}[trm]{Proposition}

\newcommand{\R}[1]{\mathbb{R}^{#1}}
\newcommand{\de}{\partial}

\newenvironment{proof}{\noindent\emph{Proof.}}{\hfill$\square$\medskip}

\newenvironment{rmk}{\medskip\noindent\emph{Remark.}}{\medskip}

\DeclareMathOperator{\loc}{loc}

\DeclareMathOperator*{\dist}{dist}
\DeclareMathOperator*{\Intm}{\int\!\!\!\!\!\! \rule[2.6pt]{6.5pt}{.4pt}}
\DeclareMathOperator{\intm}{\int\!\!\!\!\!\!\: --}

\DeclareMathOperator{\vol}{vol}

\begin{document}
\maketitle

\begin{abstract}
We discuss compactness, blow-up and quantization phenomena for the prescribed $Q$-curvature equation $(-\Delta)^m u_k=V_ke^{2mu_k}$ on open domains of $\R{2m}$. Under natural integral assumptions we show that when blow-up occurs, up to a subsequence
$$\lim_{k\to \infty}\int_{\Omega_0} V_ke^{2mu_k}dx=L\Lambda_1,$$
where $\Omega_0\subset\subset\Omega$ is open and contains the blow-up points, $L\in\mathbb{N}$ and $\Lambda_1:=(2m-1)!\vol(S^{2m})$ is the total $Q$-curvature of the round sphere $S^{2m}$. Moreover, under suitable assumptions, the blow-up points are isolated. We do not assume that $V$ is positive.
\end{abstract}

\section{Introduction}

Let $\Omega\subset\R{2m}$ be a connected open set and consider a sequence $(u_k)$ of solutions to the equation

\begin{equation}\label{eq0}
(-\Delta)^m u_k =V_k e^{2mu_k}\quad \text{in }\Omega,
\end{equation}
where
\begin{equation}\label{Vk}
V_k\to V_0 \quad \text{in } C^0_{\loc}(\Omega),
\end{equation}
and, for some $\Lambda>0$,
\begin{equation}\label{area}
\int_{\Omega} e^{2mu_k}dx\leq \Lambda.
\end{equation}

Equation \eqref{eq0} arises in conformal geometry, as it is the higher-dimensional generalization of the Gauss equation for the prescribed Gaussian curvature. In fact, if $u_k$ satisfies \eqref{eq0}, then the conformal metric
$$g_k:=e^{2u_k}|dx|^2$$
has $Q$-curvature $V_k$ (here $|dx|^2$ denotes the Euclidean metric). For the definition of $Q$-curvature and for more details about the geometric meaning of \eqref{eq0} we refer to the introduction in \cite{mar1}.

An important example of solutions to \eqref{eq0}-\eqref{area} can be constructed as follows. It is well known that the $Q$-curvature of the round sphere $S^{2m}$ is $(2m-1)!$. Then, if $\pi:S^{2m}\to \R{2m}$ is the stereographic projection, the metric $g_1:=(\pi^{-1})^*g_{S^{2m}}$ also has $Q$-curvature $(2m-1)!$. Since $g_1=e^{2\eta_0}|dx|^2$, with $\eta_0(x)=\log\frac{2}{1+|x|^2}$, it follows that
\begin{equation}\label{eta0}
\begin{split}
&(-\Delta)^m \eta_0=(2m-1)!e^{2m\eta_0},\\
&(2m-1)!\int_{\R{2m}} e^{2m\eta_0}dx=(2m-1)!\vol(S^{2m})=:\Lambda_1.
\end{split}
\end{equation}

The purpose of this paper is to study the compactness properties of \eqref{eq0}, and show analogies and differences with previous results in this direction. We start by considering the following model case. The sequence of functions $u_k(x):=\log\frac{2k}{1+k^2 |x|^2}$ satifies \eqref{eq0} on $\Omega=\R{2m}$ with $V_k\equiv (2m-1)!$ and $\int_{\R{2m}}e^{2mu_k}dx=\vol(S^{2m})$ for every $k$. On the other hand $(u_k)$ is not precompact, as $u_k(0)\to\infty$ and $u_k\to-\infty$ locally uniformly on $\R{2m}\setminus\{0\}$ so that
$$V_k e^{2mu_k}dx\rightharpoonup \Lambda_1\delta_0$$
in the sense of measures as $k\to\infty$.

\medskip

For $m=1$, Brezis and Merle in their seminal work \cite{BM} proved that a sequence $(u_k)$ of solutions to \eqref{eq0}-\eqref{area} is either bounded in $C^{1,\alpha}_{\loc}(\Omega)$, or $u_k\to-\infty$ uniformly locally in $\Omega\setminus S$, where $S=\{x^{(1)}, \ldots, x^{(I)}\}$ is a finite set. In particular one has
$$V_k dx\rightharpoonup \sum_{i=1}^I\alpha_i \delta_{x^{(i)}}$$
in the sense of measures. Brezis and Merle also conjectured that, at least for $V_0> 0$, in the latter case one has $\alpha_i=4\pi L_i$ for some positive integers $L_i$. This was shown to be true by Li and Shafrir \cite{LS}. Notice that $4\pi=\Lambda_1$ for $m=1$.

\medskip

For $m\geq 2$ things are more complex. In \cite{CC} Chang and Chen proved that for every $\alpha\in (0,\Lambda_1)$ there exists a solution $v$ to $(-\Delta)^m v=(2m-1)!e^{2mv}$ on $\R{2m}$ and with $(2m-1)!\int_{\R{2m}}e^{2mv}dx=\alpha$. Then, setting
$$u_k(x)=v(kx)+\log k,$$
we find a non-compact sequence of solutions to \eqref{eq0}, \eqref{Vk}, \eqref{area} with $V_k\equiv (2m-1)!$ and
$$\int_{\R{2m}}V_ke^{2mu_k}dx\to \alpha\not\in \Lambda_1\mathbb{N}.$$
Moreover for $m=2$ Adimurthi, Robert and Struwe \cite{ARS} gave examples of sequences $(u_k)$ with $u_k\to\infty$ on a hyperplane. These facts suggest that in order to obtain a situation similar to the results of Brezis-Merle (finiteness of the blow-up set) and of Li-Shafrir (quantization of the total $Q$-curvature), we should make further assumption. In this setting this was first done by Robert for $m=2$, and Theorem \ref{trm1} below is a generalization of Robert's result to the case when $m$ is arbitrary.

\begin{trm}\label{trm1}  Let $(u_k)\subset C^{2m}_{\loc}(\Omega)$ be solutions to \eqref{eq0}, \eqref{Vk} and \eqref{area}, and assume that there is a ball $B_{\rho}(\xi)\subset\Omega$ such that
\begin{equation}\label{Deltau}
\|\Delta u_k\|_{L^1(B_\rho(\xi))}\leq C.
\end{equation}
Then there is a finite (possibly empty) set $S=\{x^{(1)},\ldots , x^{(I)}\}$ such that one of the following is true:
\begin{itemize}
\item[(i)] up to a subsequence $u_k\to u_0$ in $C^{2m-1}_{\loc}(\Omega\setminus S)$ for some $u_0\in C^{2m}(\Omega\setminus S)$ solving $(-\Delta)^m u_0=V_0e^{2mu_0}$, or
\item[(ii)] up to a subsequence $u_k\to -\infty$ locally uniformly in $\Omega \setminus S$.
\end{itemize}
If $S\neq \emptyset$ and $V(x^{(i)})>0$ for some $1\le i \le I$, then case $(ii)$ occurs.

\medskip

Moreover, if we also assume that
\begin{equation}\label{Deltau-}
\|(\Delta u_k)^-\|_{L^1(\Omega)}\leq C,\quad \text{with } (\Delta u_k)^-:=\min\{\Delta u_k, 0\},
\end{equation}
we have in case $(i)$ that $S=\emptyset$ and in case $(ii)$ that $V_0(x^{(i)})>0$ for $1\leq i\leq I$ and
\begin{equation}\label{quant}
V_k e^{2m u_k}dx\rightharpoonup \sum_{i=1}^I \alpha_i \delta_{x^{(i)}}
\end{equation}
in the sense of measures in $\Omega$, where $\alpha_i=L_i\Lambda_1$ for some positive $L_i\in \mathbb{N}$. In particular, in case $(ii)$ for any open set $\Omega_0\subset\subset \Omega$ with $S\subset \Omega_0$ we have
\begin{equation}\label{quant2}
\int_{\Omega_0}V_k e^{2mu_k}dx\to L\Lambda_1
\end{equation}
for some $L\in\mathbb{N}$ ($L=0$ if $S=\emptyset$).
\end{trm}
Notice that the hypothesis \eqref{Deltau} and \eqref{Deltau-} are natural, since for $m=1$ they already follow from \eqref{eq0}, \eqref{Vk} and \eqref{area}, and the counterexample quoted above show that they are necessary to some extent (see the first open problem in the last section). Moreover, contrary to \cite{rob2} and \cite{LS}, we do not assume that $V_0>0$. In fact, as already discussed in \cite{mar3}, if $V_0$ has changing sign, one can show using the results of \cite{mar2} that, if \eqref{Deltau-} holds, blow-up happens only at points where $V_0>0$. We also point out that when $m=2$, F. Robert \cite{rob3} proved a version of Theorem \ref{trm1} where the assumptions \eqref{area}, \eqref{Deltau} and \eqref{Deltau-} are replaced by $\|\Delta u_k\|_{L^1(\Omega)}\leq C$. This does not seem possible for $m>2$ without further assumptions of $\Delta^j u_k$ for $2\le j\le m-1$.

\medskip

A different approach to compactness can be given by working on a closed Riemannian manifold instead of an open set, see Druet-Robert \cite{DR}, Malchiodi \cite{mal}, Martinazzi \cite{mar3} and Ndiaye \cite{ndi}, or by assuming $\Omega$ bounded and imposing a Dirichlet or a Navier boundary condition, see Wei \cite{wei}, Robert-Wei \cite{RW} and Martinazzi-Petrache \cite{MP}. In this case the quantization is even stronger, as one shows that $\alpha_i=\Lambda_1$ in \eqref{quant} and $L=I$ in \eqref{quant2}. 
It turns out that the ideas of \cite{DR} and \cite{mar3} can be applied in the present context of an open domain if we  assume an a-priori $L^1$-bound on $\nabla u_k$ in place of the bound on $\Delta u_k$:

\begin{trm}\label{trm2}  Let $(u_k)\subset C^{2m}_{\loc}(\Omega)$ be solutions to \eqref{eq0} and \eqref{area}, where 
\begin{equation}\label{Vkbis}
V_k\to V_0 \quad \text{in } C^1_{\loc}(\Omega).
\end{equation}
Assume further that there is a ball $B_{\rho}(\xi)\subset\Omega$ such that
\begin{equation}\label{nablau2}
\|\nabla u_k\|_{L^1(B_\rho(\xi))}\leq C.
\end{equation}
Then there is a finite (possibly empty) set $S=\{x^{(1)},\ldots , x^{(I)}\}$ such that one of the following is true:
\begin{itemize}
\item[(i)] up to a subsequence $u_k\to u_0$ in $C^{2m-1}_{\loc}(\Omega\setminus S)$ for some $u_0\in C^{2m}(\Omega\setminus S)$ solving $(-\Delta)^m u_0=V_0e^{2mu_0}$, or
\item[(ii)] up to a subsequence $u_k\to -\infty$ locally uniformly in $\Omega \setminus S$.
\end{itemize}
If $S\neq \emptyset$ and $V(x^{(i)})>0$ for some $1\le i \le I$, then case $(ii)$ occurs.

\medskip

Moreover, if we also assume that
\begin{equation}\label{nablau3}
\|\nabla u_k\|_{L^1(\Omega)}\leq C,
\end{equation}
we have that in case $(i)$ $S=\emptyset$ and in case $(ii)$ $V_0(x^{(i)})>0$ for $1\leq i\leq I$ and
\begin{equation}\label{quant3}
V_k e^{2m u_k}dx\rightharpoonup \sum_{i=1}^I \Lambda_1 \delta_{x^{(i)}}
\end{equation}
in the sense of measures. In particular, for any open set $\Omega_0\subset\subset \Omega$ with $S\subset \Omega_0$ we have
\begin{equation}\label{quant4}
\int_{\Omega_0}V_k e^{2mu_k}dx\to I\Lambda_1.
\end{equation}
\end{trm}

The difference between Theorem \ref{trm1} and Theorem \ref{trm2} is that under the hypothesis of Theorem \ref{trm2} one can prove uniform bounds for $\nabla^\ell u_k$, $1\leq \ell\leq 2m-2$ (Propositions \ref{gradbis} and \ref{propgrad2bis}), which in turn allow us to apply a clever technique of Druet and Robert \cite{DR} to rule out the occurrence of multiple blow-up points. In Theorem \ref{trm1} one can only prove bounds for $\nabla^{\ell-2}\Delta u_k$, $2\leq \ell\leq 2m-1$ (Propositions \ref{grad} and \ref{propgrad2} below). This is not just a technical issue, as the result of Theorem \ref{trm2} is stronger than that of Theorem \ref{trm1}. Indeed X. Chen \cite{che} showed that already for $m=1$, under the assumptions of Theorem 1, there exist sequences with multiple blow-up points.

\medskip

The paper is organized as follows. In Section \ref{sec1} we prove Theorem \ref{trm1}, in section \ref{sec2}, we prove Theorem \ref{trm2} and in the last section we collect some open problems.
The letter $C$ always denotes a generic large constant which can change from line to line, and even within the same line.

\medskip

I am grateful to F. Robert for suggesting me to work on this problems.

\section{Proof of Theorem \ref{trm1}}\label{sec1}

In the proof of Theorem \ref{trm1} we use the strategy of extracting blow-up profiles (Proposition \ref{exh} below), in the spirit of Struwe \cite{str1}, \cite{str2} and of Br\'ezis-Coron \cite{BC1}, \cite{BC2}. We classify such profiles thanks to the results of \cite{mar1} and \cite{mar2}, and finally we use Harnack-type estimates inspired from \cite{rob2}. Since Propositions \ref{prop4} and \ref{grad} below don't work for $m=1$, in this section we shall assume that $m>1$. For the case $m=1$ we refer to \cite{LS}, noticing that their assumption $V_k\ge 0$ can be easily dropped (particularly in their Lemma 1), since there are no solutions to the equation
$$-\Delta u=Ve^{2u}\text{ in }\R{2},\quad \int_{\R{2}}e^{2u}dx<\infty, \quad V\equiv const<0,$$
see Theorem 1 in \cite{mar2}.
 
\begin{prop}\label{prop3} Let $(u_k)$ be a sequence of solutions to \eqref{eq0}-\eqref{area} satisfying \eqref{Deltau} for some ball $B_\rho(\xi)\subset\Omega$ and set
\begin{equation}\label{defS}
S:=\left\{ y\in \Omega:\lim_{r\to 0^+}\liminf_{k\to\infty}\int_{B_r(y)}|V_k| e^{2mu_k}dy\geq \frac{\Lambda_1}{2}   \right\}.
\end{equation}
Then $S$ is finite (possibly empty) and up to selecting a subsequence one of the following is true:
\begin{itemize}
\item[(i)] $u_k\to u_0$ in $C^{2m-1}_{\loc}(\Omega\backslash S)$ for some $u_0\in C^{2m}(\Omega\backslash S)$;
\item[(ii)] $u_k\to -\infty$ locally uniformly in $\Omega \backslash S$.
\end{itemize}
If $S\neq \emptyset$ and $V(x^{(i)})>0$ for some $1\le i \le I$, then case $(ii)$ occurs.
\end{prop}

\begin{proof}
By Theorem 1 in \cite{mar3} (compare \cite{ARS}) we have that $S$ is finite and either
\begin{itemize}
\item[$(a)$] $u_k\to u_0$ in $C^{2m-1}_{\loc}(\Omega\backslash S)$ for some $u_0\in C^{2m}(\Omega\backslash S)$, or
\item[$(b)$] $u_k\to -\infty$ locally uniformly in $\Omega \backslash (S\cup \Gamma)$, where $\Gamma$ is a closed set of Hausdorff dimension at most $2m-1$. Moreover there are numbers $\beta_k\to \infty$ such that
\begin{equation}\label{ubeta}
\frac{u_k}{\beta_k}\to \varphi\quad \text{in } C^{2m-1}_{\loc}(\Omega \backslash (S\cup \Gamma)),
\end{equation}
where $\varphi\in C^\infty(\Omega\backslash S)$, $\Gamma=\{x\in \Omega\setminus S:\varphi(x)=0\}$ and
\begin{equation}\label{phi}
\Delta^m \varphi\equiv 0,\quad \varphi\leq 0,\quad \varphi \not\equiv 0 \quad \text{in }\Omega\backslash S.
\end{equation}
\end{itemize}
Clearly case $(a)$ corresponds to case $(i)$ in the proposition. We need to show that if $(b)$ occurs, then $\Gamma=\emptyset$, so that $\varphi<0$ on $\Omega\backslash S$ and case $(ii)$ follows from \eqref{ubeta}. In order to show that $\Gamma=\emptyset$, observe that $\Delta\varphi\equiv 0$ in $\Omega\backslash S$. Otherwise, since $\Delta \varphi$ is analytic\footnote{we have $\Delta^{m-1}(\Delta\varphi)=0$, and polyharmonic functions are analytic.}, we would have
$$\int_{B_\rho(\xi)}|\Delta\varphi|dx>0,$$
where $B_\rho(\xi)\subset\Omega$ is as in \eqref{Deltau}. Then \eqref{ubeta} would imply
$$\lim_{k\to\infty}\int_{B_\rho(\xi)}|\Delta u_k|dx=\lim_{k\to\infty}\beta_k\int_{B_\rho(\xi)}|\Delta \varphi|dx=+\infty,$$
contradicting \eqref{Deltau}. Therefore $\Delta\varphi\equiv 0$. Then the maximum principle and \eqref{phi} imply that $\varphi<0$ in $\Omega\backslash S$, i.e. $\Gamma=\emptyset$, as wished. Also the last claim follows from Theorem 1 in \cite{mar3}.
\end{proof}

Proposition \ref{prop3} completes the proof of the first part of Theorem \ref{trm1}. In the remaining part of this section we shall assume that $(u_k)$ satisfies all the hypothesis of Theorem \ref{trm1}, including \eqref{Deltau-} in particular, and we shall prove the second part of Theorem \ref{trm1}. If $S=\emptyset$, it is clear that the proof of Theorem \ref{trm1} is complete. Therefore we shall also assume that $S\neq \emptyset$, and we shall prove that consequently we are in case $(ii)$ of Theorem 1.

\begin{prop}\label{prop4}
For every open set $\Omega_0\subset\subset \Omega\backslash S$ there is a constant $C(\Omega_0)$ independent of $k$ such that
\begin{equation}\label{stime}
\|\Delta u_k\|_{C^{2m-3}(\Omega_0)}\leq C(\Omega_0).
\end{equation}
\end{prop}
\begin{proof}
If case $(i)$ of Proposition \ref{prop3} occurs the proof of \eqref{stime} is trivial, hence we shall assume that we are in case $(ii)$. Up to restricting the ball $B_\rho(\xi)$ given in \eqref{Deltau}, we can assume that $B_{2\rho}(\xi)\cap S=\emptyset$, so that $u_k\leq C=C(\rho)$ on $B_\rho(\xi)$. Consequently $|\Delta^m u_k|\leq C$ on $B_\rho(\xi)$. This, \eqref{Deltau} and elliptic estimates (see e.g. \cite{mar1}, Lemma 20) imply that
\begin{equation}\label{5bis}
\|\Delta u_k\|_{C^{2m-3}(B_{\rho/2}(\xi))}\leq C.
\end{equation}
Elliptic estimates and \eqref{Deltau-} imply that either $\Delta u_k\to +\infty$ locally uniformly in $\Omega\setminus S$, or $(\Delta u_k)_{k\in\mathbb{N}}$ is uniformly bounded locally in $\Omega\setminus S$. In the first case \eqref{5bis} cannot hold, so we are in the second situation, and \eqref{stime} follows at once from elliptic estimates, since $|\Delta^m u_k|\le C(\Omega_0)$ on $\Omega_0$.
\end{proof}

\begin{prop}\label{grad} For every open set $\Omega_0\subset\subset \Omega$ there is a constant $C$ independent of $k$ such that 
\begin{equation}\label{grad2}
\int_{B_r(x_0)}|\nabla^{\ell-2}\Delta u_k|dx\leq C r^{2m-\ell},
\end{equation}
for $2\leq \ell \leq 2m-1$ and for every ball $B_r(x_0)\subset \Omega_0$.
\end{prop}

\begin{proof} Fix
$$\delta=\frac{1}{16}\min\left\{\min_{1\leq i\neq j\leq I} |x^{(i)}-x^{(j)}|,\dist(\de\Omega,\de\Omega_0) \right\}.$$
By a covering argument, it is enough to prove \eqref{grad2} for $0<r\leq \delta$. Given $B_r(x_0)\subset\Omega_0$ with $r\le \delta$, we can choose a ball $B_{4\delta} (\xi)\subset\Omega$ such that $B_r(x_0)\subset B_{2\delta}(\xi)$, $\dist(\de B_{4\delta}(\xi),S)\geq 2\delta$.
For $x\in B_{2\delta}(\xi)$, let $G_x(y)$ be the Green function for the operator $\Delta^{m-1}$ in $B_{4\delta}(\xi)$ with respect to the Navier boundary condition:
$$\Delta^{m-1}G_x=\delta_x \text{ in }B_{4\delta}(\xi),\quad G_x =\Delta G_x=\ldots=\Delta^{m-2}G_x=0\text{ on }\de B_{4\delta}(\xi).$$ 
Then we can write
\begin{equation}\label{poisson}
\begin{split}
\Delta u_k(x)=&\int_{B_{4\delta}(\xi)} G_x(y)\Delta^{m-1}\Delta u_k(y)dy\\
&+\sum_{j=0}^{m-2}\int_{\de B_{4\delta}(\xi)}\frac{\de}{\de\nu}(\Delta^{m-j-2}G_x) \Delta^j(\Delta u_k)d\sigma.
\end{split}
\end{equation}
Differentiating and using the bound $|\nabla^{\ell-2}G_x(y)|\leq \frac{C}{|x-y|^\ell}$ (see \cite{DAS}) and \eqref{stime} on $\de B_{4\delta}(\xi)$, we infer for $x\in B_{2\delta}(\xi)$
\begin{equation}\label{nablau}
\begin{split}
|\nabla^{\ell-2}\Delta u_k(x)|\leq & C\int_{B_{4\delta}(\xi)}\frac{|V_k(y)|e^{2mu_k(y)}}{|x-y|^\ell}dy\\
&+ C \sum_{j=0}^{m-2}\sup_{\de B_{4\delta}(\xi)}\Big(\Delta^{j+1}u_k \Big)\int_{\de B_{4\delta}(\xi)}\frac{d\sigma(y)}{|x-y|^{\ell+2m-2j-3}}\\
\le & C\int_{B_{4\delta}(\xi)}\frac{e^{2mu_k(y)}}{|x-y|^\ell}dy+C.
\end{split}
\end{equation}
Integrating on $B_r(x_0)$ and using Fubini's theorem, we finally get
\begin{equation*}
\begin{split}
\int_{B_r(x_0)}|\nabla^{\ell-2}\Delta u_k(x)|dx\leq& C\int_{B_r(x_0)}\int_{B_{4\delta}(\xi)}\frac{e^{2mu_k(y)}}{|x-y|^\ell}dydx +Cr^{2m}\\
\leq &C\int_{B_{4\delta}(\xi)}e^{2mu_k(y)}\bigg(\int_{B_r(x_0)}\frac{1}{|x-y|^\ell}dx\bigg)dy+Cr^{2m}\\
\leq & C r^{2m-\ell}\int_{B_{4\delta}(\xi)}e^{2mu_k(y)}dy+Cr^{2m}\\
\leq & Cr^{2m-\ell}+Cr^{2m}\le Cr^{2m-\ell},
\end{split}
\end{equation*}
where in the last inequality we used that $r\le \delta$.
\end{proof}

\begin{prop}\label{exh} Let $\Omega_0\subset\subset\Omega$ be an open set such that $S\subset\Omega_0$. Then up to a subsequence we have
\begin{equation}\label{supinf}
\lim_{k\to\infty}\sup_{\Omega_0}u_k=+\infty,
\end{equation}
and case $(ii)$ of Theorem \ref{trm1} occurs.
There exist $L\ge I$ converging sequences of points $x_{i,k}\to x^{(i)}\in\Omega$ such that $u_k(x_{i,k})\to\infty$ as $k\to\infty$, $S=\{x^{(1)},\ldots, x^{(L)}\}$, $V(x^{(i)})>0$ for $1\leq i\leq L$, and there exist $L$ sequences of positive numbers
\begin{equation}\label{mu}
\mu_{i,k}:=2\bigg(\frac{(2m-1)!}{V_0(x^{(i)})}\bigg)^\frac{1}{2m}e^{-u_k(x_{i,k})}\to 0
\end{equation}
such that the following holds:
\begin{itemize}
\item[(a)] for $1\leq i,j\leq L$, $i\neq j$
$$\lim_{k\to\infty}\frac{|x_{i,k}-x_{j,k}|}{\mu_{i,k}}=\infty;$$
\item[(b)] setting $\eta_{i,k}:=u_k(x_{i,k}+\mu_{i,k}x)-u_k(x_{i,k})+\log 2$, one has
$$\lim_{k\to\infty} \eta_{i,k}(x)=\eta_0(x)=\log\frac{2}{1+|x|^2}\quad \text{in }C^{2m-1}_{\loc}(\R{2m}),$$
and
\begin{equation}\label{energia}
\lim_{R\to\infty}\lim_{k\to\infty}\int_{B_{R\mu_{i,k}}(x_{i,k})}V_ke^{2mu_k}dx=\Lambda_1;
\end{equation}
\item[(c)] for every $\Omega_0\subset\subset\Omega$ we have
\begin{equation}\label{stima}
\inf_{1\leq i\leq L}|x-x_{i,k}|e^{u_k(x)}\leq C= C(\Omega_0).
\end{equation}
\end{itemize}
\end{prop}

\begin{proof} \emph{Step 1.} If $\sup_{\Omega_0}u_k\leq C$, then by \eqref{defS} we have $S=\emptyset$, contrary to the assumption we made after Proposition \ref{prop3}. Therefore we can assume that \eqref{supinf} holds.

\medskip

\noindent \emph{Step 2.}
Since $u_k$ is locally bounded in $\Omega\setminus S$ uniformly in $k$ if case $(i)$ of Theorem \ref{trm1} holds, and $u_k\to-\infty$ uniformly locally in $\Omega\setminus S$, one can find $x_k\in\Omega_0$ such that
$$u_k(x_k)=\sup_{\Omega_0}u_k\to \infty, \quad \text{as }k\to\infty.$$
Moreover up to a subsequence $x_k\to x_0\in S$. In particular $\dist(x_k,\de\Omega_0)\geq \delta>0$ for some $\delta>0$. 
Setting $\sigma_k=e^{-u_k(x_k)}$, we define
$$z_k(y)=u_k(x_k+\sigma_k y)+\log(\sigma_k)\leq 0\quad\text{in } B_{\delta/\sigma_k}(0).$$
We claim that up to a subsequence $z_k\to z_0$ in $C^{2m-1,\alpha}_{\loc}(\R{2m})$, where
\begin{equation}\label{z0}
(-\Delta)^m z_0=V_0(x_0)e^{2mz_0},\quad \int_{\R{2m}}e^{2mz_0}dx<\infty.
\end{equation}
This follows by elliptic estimates, using that $z_k\leq 0$, $z_k(0)=0$ and Proposition \ref{grad}.
With the same technique of the proof of Proposition 8 in \cite{mar3}, step 3, one proves that $V_0(x_0)>0$.
Since we have found a point $x_0\in S$ with $V_0(x_0)>0$, Proposition \ref{prop3} implies that we are in case $(ii)$ of Theorem \ref{trm1}.

\medskip

\noindent\emph{Step 3.}
Now we define $x_{1,k}:=x_k\to x_0=:x^{(1)}$. Also set $\mu_{1,k}$ and $\eta_{1,k}$ as in the statement of the proposition. Then, still following \cite{mar3}, Proposition 8, we infer that $\eta_{1,k}(x)\to \log\frac{2}{1+|x|^2}$ in $C^{2m-1,\alpha}_{\loc}(\R{2m})$. 

\medskip

\noindent\emph{Step 4.} We now proceed by induction, as follows. Assume that we have already found $L$ sequences $(x_{i,k})$ and $(\mu_{i,k})$, $1\leq i\leq L$, such that (a) and (b) holds, we either have that also $(c)$ holds, and we are done, or we construct a new sequence $x_{L+1,k}=x_k\to x_0\in S$, and $\sigma_k=\sigma_{L+1,k}:=e^{-u_k(x_k)}$
such that
$$\inf_{1\leq i\leq L}|x_k-x^{(i)}|e^{u_k(x_k)}=\max_{x\in\Omega_0}\inf_{1\leq i\leq L}|x-x^{(i)}|e^{u_k(x)}.$$
Then we define $z_k\to z_0$ as before, we prove that $V_0(x_0)>0$, so that we can define $\mu_{L+1,k}$ and $\eta_{L+1,k}$ as in the statement of the proposition and $\eta_{L+1,k}(x)\to\log\frac{2}{1+|x|^2}$ in $C^{2m-1,\alpha}_{\loc}(\R{2m})$. Moreover (a) holds with $L+1$ instead of $L$. Taking into account (a) and (b), we see that
$$\limsup_{k\to\infty}\int_{\Omega_0}V_ke^{2mu_k}dx\geq (L+1)\Lambda_1.$$
This, \eqref{Vk} and \eqref{area} imply that after a finite number of steps the procedure stops and (c) holds. The missing details are as in Step 1 of the proof of Theorem 1 in \cite{DR}.
\end{proof}

\begin{rmk} In general, as shown by X. Chen \cite{che}, it is possible that $L>I$, hence $x^{(i)}=x^{(j)}$ for some $i\neq j$. In this case we will stick to the notation $S=\{x^{(i)},\ldots,x^{(I)}\}$, i.e. $x^{(i)}\neq x^{(j)}$ for $i\neq j$, $1\le i,j\le I$.
\end{rmk}

\begin{prop}\label{propgrad2} For $2\leq \ell\leq 2m-1$ and $\Omega_0\subset\subset\Omega$ we have
\begin{equation}\label{deltau}
\inf_{1\leq i\leq L} |x-x_{i,k}|^{\ell} |\nabla^{\ell-2}\Delta u_k(x)|\leq C=C(\Omega_0), \quad \text{for }x\in\Omega_0.
\end{equation}
\end{prop}

\begin{proof} Let us consider a ball $B_\delta(\xi)$ as in the proof of Proposition \ref{grad}, so that we have
$$|\nabla^{\ell-2}\Delta u_k(x)|\le C\int_{B_{4\delta}(\xi)}\frac{e^{2mu_k(y)}}{|x-y|^\ell}dy+C$$
for $x\in B_{2\delta}(\xi)$ which we now fix.
Set for $1\leq i\leq L$
$$\Omega_{i,k}:=\Big\{y\in B_{2\delta}(\xi):\inf_{1\leq j\leq L} |y-x_{j,k}| =|y-x_{i,k}|\Big\},$$
and, assuming $x\neq x_{i,k}$ for $1\leq i\leq L$ (otherwise \eqref{deltau} is trivial), set
$$\Omega_{i,k}^{(1)}:=\Omega_{i,k}\cap B_{|x_{i,k}-x|/2}(x_{i,k}),\quad \Omega_{i,k}^{(2)}:=\Omega_{i,k}\backslash B_{|x_{i,k}-x|/2}(x_{i,k}).$$
Observing that for $y\in \Omega_{i,k}^{(1)}$ we have $\frac{1}{|x-y|}\leq \frac{2}{|x-x_{i,k}|}$ and using (c) from Proposition \ref{exh}, we infer
\begin{eqnarray*}
\int_{\Omega_{i,k}}\frac{e^{2mu_k}}{|x-y|^\ell}dy&\leq& \frac{C}{|x-x_{i,k}|^\ell}\int_{\Omega_{i,k}^{(1)}}e^{2mu_k(y)}dy\\
&&+C\int_{\Omega_{i,k}^{(2)}}\frac{dy}{|x-y|^\ell |y-x_{i,k}|^{2m}}.
\end{eqnarray*}
The first integral on the right-hand side is bounded by $\frac{C}{|x- x_{i,k}|^\ell}.$ As for the integral over $\Omega_{i,k}^{(2)}$, write $\Omega_{i,k}^{(2)}=\Omega_{i,k}^{(3)}\cup \Omega_{i,k}^{(4)}$, with
$$\Omega_{i,k}^{(3)}= \Omega_{i,k}^{(2)}\cap B_{2|x- x_{i,k}|}(x),\quad \Omega_{i,k}^{(4)}=\Omega_{i,k}^{(2)}\backslash B_{2|x- x_{i,k}|}(x).$$
We have
\begin{eqnarray*}
\int_{\Omega_{i,k}^{(3)}}\frac{dy}{|x-y|^\ell |y-x_{i,k}|^{2m}}&\leq& \frac{C}{|x-x_{i,k}|^{2m}}\int_{\Omega_{i,k}^{(3)}}\frac{dy}{|x-y|^\ell}\\
&\le&\frac{C}{|x-x_{i,k}|^{2m}}\int_0^{2|x-x_{i,k}|}r^{2m-\ell-1}dr\\
&\leq& \frac{C}{|x-x_{i,k}|^\ell}.
\end{eqnarray*}
Observing that
$$\frac{1}{C}|y-x_{i,k}|\leq |x-y|\leq C|y-x_{i,k}|\quad \textrm{on }\Omega_{i,k}^{(4)},$$
we estimate
\begin{eqnarray*}
\int_{\Omega_{i,k}^{(4)}}\frac{dy}{|x-y|^\ell |y-x_{i,k}|^{2m}}&\leq& C \int_{\Omega_{i,k}^{(4)}}\frac{dy}{|x-y|^{2m+\ell}}\\
&\leq& C\int_{2|x-x_{i,k}|}^\infty r^{-\ell-1}dr\\
&\leq& \frac{C}{|x-x_{i,k}|^{\ell}}.
\end{eqnarray*}
Putting these inequalities together yields
$$|\nabla^{\ell-2}\Delta u_k(x)|\leq \frac{C}{\inf_{1\leq i\leq L}|x-x_{i,k}|^\ell}+C.$$
This gives \eqref{deltau} for $x\in B_{2\delta}(\xi)\setminus S$ and for $\dist(x,S)\leq 1$. For $\dist(x,S)\ge 1$, \eqref{deltau} follows from Proposition \ref{prop4}. By a simple covering argument, we conclude.
\phantom{ } \end{proof}

Analogous to Proposition 4.1 in \cite{rob2} we have the following result, which is the key step in showing that the contributions given by \eqref{energia} for $1\le i\le L$ asymptotically exhaust the whole energy.

\begin{prop}\label{prop4.1}
Consider $x_0\in S$, $0<\delta<\frac{\dist(x_0,\de\Omega)}{4}$, such that $V_k(x)\ge V_k(x_0)/2>0$ for $x\in B_{4\delta}(x_0)$ and $k$ large enough. Up to relabelling assume that
$$\lim_{k\to\infty} x_{i,k}=x_0,\quad  \text{for }1\leq i\leq N,$$
for some positive integer $N\le L$, and set $x_k:=x_{1,k}$, $\mu_k:=\mu_{1,k}$. Assume that for a sequence $0\le \rho_k\to 0$ we have
\begin{equation}\label{4.4}
\inf_{1\le i\le N}|x-x_{i,k}|e^{u_k(x)}\le C,\quad \inf_{1\le i\le N}|x-x_{i,k}|^\ell|\nabla^{\ell-2}\Delta u_k(x)|\le C
\end{equation}
for $x\in B_{2\delta(x_k)}\setminus B_{\rho_k}(x_k)$ and $2\leq \ell\leq 2m-1$.
Let $r_k>0$ be such that $r:=\lim_{k\to\infty}r_k\in[0,\delta]$, $\lim_{k\to\infty}\frac{\mu_k}{r_k}=\lim_{k\to\infty}\frac{\rho_k}{r_k}=0$ and set
$$J:=\Big\{i\in\{2,\ldots,N\}:\limsup_{k\to\infty}\frac{|x_{i,k}-x_k|}{r_k}<\infty \Big\}.$$
Up to a subsequence, define $\tilde x_i:=\lim_{k\to\infty}\frac{x_{i,k}-x_k}{r_k}$, for $i\in J$.
Assume that $\tilde x_i\neq 0$ for $i\in J$ and let $\nu$ and $R$ be such that
\begin{equation}\label{4.9}
0<\nu<\frac{1}{10}\min\big\{ \{|\tilde x_i|:i\in J\}\cup \{|\tilde x_i-\tilde x_j|:i,j\in J, \tilde x_i\neq\tilde x_j\}   \big\},
\end{equation}
and
\begin{equation}\label{4.10}
3\max \{|\tilde x_i|: i\in J\}<R<\frac{\delta}{2r},
\end{equation}
where $\frac{\delta}{2r}:=\infty$ if $r=0$.
Then we have
\begin{equation}\label{stima3}
\lim_{k\to\infty}\int_{\big(B_{Rr_k}(x_k)\setminus \bigcup_{i\in J}\overline{B}_{\nu r_k}(x_{i,k})\big)\setminus\overline{B}_{3\rho_k}(x_k)}e^{2mu_k}dx=0,\quad \text{if }\mu_k/\rho_k\to 0,
\end{equation}
as $k\to \infty$, and
\begin{equation}\label{stima4}
\lim_{\tilde R\to\infty}\lim_{k\to\infty}\int_{\big(B_{Rr_k}(x_k)\setminus \bigcup_{i\in J}\overline{B}_{\nu r_k}(x_{i,k})\big)\setminus\overline{B}_{\tilde R\mu_k}(x_k)}e^{2mu_k}dx=0, \quad \text{if }\rho_k\leq C\mu_k.
\end{equation}
\end{prop}

\begin{rmk}
For a better understanding of the above proposition one can first consider the simplified case when $N=L=1$ (only one blow-up sequence), $r_k=\delta$, $\rho_k=0$, $R=\frac{1}{4}$ and $J=\emptyset$. Then \eqref{stima4} reduces to
$$\lim_{\tilde R\to\infty}\lim_{k\to\infty}\int_{B_{\delta/4}(x_k)\setminus \overline B_{\tilde R\mu_k}(x_k)}e^{2mu_k}dx=0. $$
This and \eqref{energia} imply \eqref{quant} with $\alpha_1=\Lambda_1$, hence the proof of Theorem \ref{trm1} is complete in this special case. 

In the general case we point out that the estimates in \eqref{4.4} are stronger than \eqref{stima} and \eqref{deltau} in that the infimum is not taken over all $1\leq i\leq L$, and weaker in that they need not hold in $B_{\rho_k}(x_k)$.
\end{rmk}

\begin{proof} First observe that if $\rho_k\le C\mu_k$,  upon redefining $\rho_k$ larger, we see that \eqref{stima3} implies \eqref{stima4}, hence we shall assume that $\lim_{k\to\infty}\mu_k/\rho_k=0$.

\medskip

\noindent \emph{Step 1.} Set
$$\Omega_k:=\bigg(B_{3R}(0)\setminus \bigcup_{i\in J} \overline B_\frac{\nu}{2}(\tilde x_i)\bigg)\setminus \overline B_{\frac{\rho_k}{r_k}}(0).$$
Then, as in \cite{rob2}, we easily get that for $x\in\Omega_k$ and $k$ large enough
\begin{equation}\label{4.16}
\inf_{1\leq i\leq N}|x_k+r_kx-x_{i,k}|\ge C(\nu,R)r_k|x|,
\end{equation}
and
$$x_k+r_kx\in B_{2\delta}(x_k)\setminus \overline B_{\rho_k}(x_k).$$
Set $\tilde u_k(x):=u_k(x_k+r_k x)+\log r_k$ for $x\in B_{3R}(0)$, satisfying
$$(-\Delta)^m \tilde u_k=\tilde V_k e^{2m \tilde u_k}\quad \text{in } B_{3R}(0)$$
for $\tilde V_k(x):=V_k(x_k+r_kx)$. According to \eqref{4.4} we have
\begin{equation}\label{4.19}
|x|e^{\tilde u_k(x)}\leq C,\quad |x|^{2\ell}|\Delta^\ell\tilde u_k(x)|\leq C\quad \text{for }x\in \Omega_k, \;1\leq \ell\leq m-1.
\end{equation}

\medskip

\noindent\emph{Step 2.} There are constants $C=C(\nu,R)$, $\beta=\beta(\nu,R)>0$ such that
\begin{equation}\label{4.20}
\sup_{\substack{|x|=r \\ x\not\in\bigcup_{i\in J}\overline B_{\nu}(\tilde x_i)}}(\beta\tilde u_k(x)) \le \inf_{\substack{|x|=r \\ x\not\in\bigcup_{i\in J}\overline B_{\nu}(\tilde x_i)}} \tilde u_k(x) +(1-\beta)\log r+C,
\end{equation}
for all $r\in]3\rho_k/r_k,2R]$. This follows exactly as in step 4.2 of \cite{rob2},  using \eqref{4.19} and Harnack's inequality.

\medskip

\noindent\emph{Step 3.} We claim that there exists $\alpha>0$ such that
\begin{equation}\label{4.27}
\sup_{\substack{|x|=r \\ x\not\in\bigcup_{i\in J}\overline B_{\nu}(\tilde x_i)}}\tilde u_k(x)\le -(1+\alpha)\log r -\alpha\log\frac{r_k}{\mu_k}+C
\end{equation}
for all $r\in]3\rho_k/r_k,2R]$. In order to prove this claim, fix $s_k\in ]3\rho_k/r_k,2R]$ and set
$$U_k(x):=\tilde u_k(s_kx)+\log s_k\quad \text{for }x\in B_\frac{3R}{s_k}(0).$$
Assume that $0<s_k<8\nu$, so that
$$B_1(0)\cap \bigg(\bigcup_{i\in J} \overline B_{\frac{2\nu}{s_k}}(s_k^{-1}\tilde x_i)\bigg)=\emptyset,$$
and let $H$ be the Green's function of $\Delta^m$ on $B_1$ with Navier boundary condition, that is the only function satisfying
$$\Delta^m H=\delta_0 \text{ on }B_1,\quad \quad H=\Delta H=\cdots =\Delta^{m-1}H=0\text{ on }\de B_1.$$
Then we have
\begin{equation}\label{4.32}
U_k(0)=\int_{B_1} H(y)\Delta^m U_k(y)dy+\sum_{\ell=0}^{m-1}\int_{\de B_1}\frac{\de \Delta^{m-1-\ell} H(y)}{\de n}\Delta^\ell U_k(y)d\sigma(y).
\end{equation}
Using \eqref{4.9} and \eqref{4.10} we infer that $\de B_1\subset s_k^{-1}\Omega_k$. Moreover \eqref{4.19} yields
$$U_k(x)\leq C,\quad |\Delta^\ell U_k(x)|\leq C\quad \text{for }|x|=1,\;1\leq \ell\leq m-1.$$
This implies
$$ \bigg|\int_{\de B_1}\frac{\de \Delta^{m-1-\ell} H(y)}{\de n}\Delta^\ell U_k(y)d\sigma(y)\bigg|\le C,\quad \text{for }1\leq \ell\leq m-1,$$
and
$$\int_{\de B_1}\frac{\de \Delta^{m-1} H(y)}{\de n} U_k(y)d\sigma(y) \geq \inf_{\de B_1}U_k,$$
where we used the identity $\int_{\de B_1}\frac{\de \Delta^{m-1} H(y)}{\de n} d\sigma(y)=1$. This in turn can be checked by testing \eqref{4.32} with $U_k\equiv 1$. Then, also observing that $(-1)^mH\geq 0$ and $(-\Delta)^m U_k\geq 0$, \eqref{4.32} gives
\begin{equation}\label{4.32bis}
\begin{split}
U_k(0)&\ge \int_{B_1} (-1)^mH(y)(-\Delta)^m U_k(y)dy+\inf_{\de B_1}U_k-C\\
&\ge \int_{B_\frac{\tilde R\mu_k}{s_kr_k}} (-1)^mH(y)(-\Delta^m) U_k(y)dy+\inf_{\de B_1}U_k-C,
\end{split}
\end{equation}
for any $\tilde R>0$ and $k\geq k_0$ such that $B_\frac{\tilde R\mu_k}{s_kr_k}\subset B_\frac{1}{2}$.
We have that
\begin{equation}\label{stimaH}
(-1)^m H(y)\ge \frac{2}{\Lambda_1}\log\frac{1}{|y|}-C,
\end{equation}
which follows by elliptic estimates and the fact that $K(x):=\frac{2}{\Lambda_1}\log\frac{1}{|x|}$ satisfies $(-\Delta)^m K=\delta_0$ (see e.g. \cite[Proposition 22]{mar1}), hence $\Delta^m((-1)^m K-H)=0$. Plugging \eqref{stimaH} into \eqref{4.32bis} we can further estimate
\begin{equation}\label{4.32ter}
U_k(0)-\inf_{\de B_1}U_k+C\ge \int_{B_\frac{\tilde R\mu_k}{s_kr_k}} \bigg(\frac{2}{\Lambda_1}\log\frac{1}{|y|}-C\bigg)(-\Delta)^m U_k(y)dy=:I.
\end{equation}
Scaling back, recalling that $u_k(x_k)=-\log\mu_k + \frac{1}{2m}\log\frac{(2m-1)!}{V_0(x_0)}$, and performing the change of variable $y=\frac{\mu_k}{s_k r_k}z$, we obtain
\begin{equation*}
\begin{split}
I&= \int_{B_\frac{\tilde R\mu_k}{s_kr_k}} \bigg(\frac{2}{\Lambda_1}\log\frac{1}{|y|}-C\bigg) V_k(x_k+r_ks_k y)e^{2mU_k(y)}dy\\
&= \int_{B_{\tilde R}}\frac{2}{\Lambda_1}\bigg(\log\frac{1}{|z|}+\log\frac{s_kr_k}{\mu_k}-C\bigg)\frac{(2m-1)!V_k(x_k+\mu_kz)}{V_0(x_0)}e^{2m\eta_k}dz,
\end{split}
\end{equation*}
with $\eta_k=\eta_{1,k}$ is as in Proposition \ref{exh}, part $b$.
Then Proposition \ref{exh} implies for $k\geq k_0(\tilde R)$
$$I\geq (1+o(1))\frac{2}{\Lambda_1}\log\frac{s_kr_k}{\mu_k}\int_{B_{\tilde R}}(2m-1)!e^{2m\eta_0}dz,$$
with error $o(1)\to 0$ as $k\to\infty$. Then with \eqref{eta0} we get
$$I\geq (2+\theta_k(\tilde R))\log\frac{s_kr_k}{\mu_k}$$
for some function $\theta_k(\tilde R)$ with $\lim_{\tilde R\to\infty}\lim_{k\to\infty}\theta_k(\tilde R)=0$.
Going back to \eqref{4.32ter} and observing that $U_k(0)=\log\frac{r_ks_k}{\mu_k}+C$, we conclude
$$(1+\theta_k(\tilde R))\log\frac{s_kr_k}{\mu_k}+\inf_{\de B_1}U_k\leq C,$$
for $k\geq k_0(\tilde R)$ large enough.
Upon choosing $\tilde R$ large, we see that there exists $\theta>-1$ such that
$$(1+\theta)\log\frac{s_kr_k}{\mu_k}+\inf_{\de B_1}U_k\leq C$$
for all $k$ large enough.
Combining this with \eqref{4.20} we obtain \eqref{4.27} with $\alpha:=\frac{1+\theta}{\beta}>0$, at least under the assumption that $r<8\nu$. For $r\geq 8\nu$ \eqref{4.27} follows from the case $r=7\nu$ and \eqref{4.20}.

\medskip

\noindent\emph{Step 4.} We now complete the proof of \eqref{stima3}. For $y\in B_R(0)\setminus \bigcup_{i=1}^\ell \overline{B}_{\nu/2}(\tilde x_i)$ we get from \eqref{4.27} (upon taking $\nu$ smaller)
$$\tilde u_k(y)\leq -(1+\alpha)\log|y|-\alpha\log\frac{r_k}{\mu_k}+C.$$
Finally, scaling back to $u_k$ and observing that $\overline B_{\nu/2}(\tilde x_i)\subset \overline B_\nu\Big(\frac{x_{i,k}-x_k}{r_k}\Big)$ for $k$ large enough, one gets
\begin{equation*}
\begin{split}
&\int_{\big(B_{Rr_k}(x_k)\setminus \bigcup_{i\in J}\overline{B}_{\nu r_k}(x_{i,k})\big)\setminus \overline B_{3\rho_k}(x_k)}e^{2mu_k}dx\\
&\qquad \le \int_{\big(B_R\setminus\bigcup_{i\in J} \overline{B}_{\frac{\nu}{2}}(\tilde x_i)\big)\setminus \overline B_{\frac{3\rho_k}{r_k}}}  e^{2m\tilde u_k(y)}dy \\
&\qquad \le\int_{\R{2m}\setminus \overline B_{\frac{3\rho_k}{r_k}}}C\bigg(\frac{\mu_k}{r_k}\bigg)^{2m\alpha}\frac{1}{|y|^{2m(1+\alpha)}}dy\\
&\qquad \le C\bigg(\frac{\mu_k}{\rho_k}\bigg)^{2m\alpha}\to 0,\quad \text{as }k\to\infty.
\end{split}
\end{equation*}
\end{proof}

Finally we claim that for any $N>0$ the following proposition holds.

\begin{prop}\label{propHN} Given a ball $B_{4\delta}(x_0)\subset\R{2m}$, let $(u_k)\subset C^{2m}(B_{4\delta(x_0)})$ be a sequence of solutions to \eqref{eq0}, \eqref{Vk}, \eqref{area} with $\Omega=B_{4\delta}(x_0)$, $V_k\geq V_0(x_0)/2>0$. Let $x_{i,k}$ and $\mu_{i,k}$, $1\leq i\leq L$ be as in Proposition \ref{exh}, and assume that $1\le L\le N$, and $\lim_{k\to\infty}x_{i,k}=x_0$ for $1\le i\le L$. Then
$$\lim_{k\to\infty}\int_{B_\delta(x_0)}V_ke^{2mu_k}dx=L\Lambda_1.$$
\end{prop}

The proof of Proposition \ref{propHN} follows from Proposition \ref{prop4.1} and \eqref{energia} by induction on $N$ as in \cite{rob2}, Proposition $(H_N)$, with only minor and straightforward modifications.

\medskip

\noindent\emph{Proof of Theorem \ref{trm1}.} Fix $\Omega_0\subset\subset\Omega$ open with $S\subset \Omega_0$ and choose $\delta>0$ such that $B_{4\delta}(x^{(i)})\subset\Omega_0$ for $1\leq i\leq I$ and and $B_{4\delta}(x^{(i)})\cap B_{4\delta}(x^{(j)})=\emptyset$ for $1\le i\ne j\le I$ (remember that $x^{(i)}\neq x^{(j)}$ for $1\le i\ne j\le I$) and such that $V_k\ge V_k(x^{(i)})/2>0$ on $B_{4\delta}(x^{(i)})$ for $k$ large enough and $1\le i\le I$. We fix $i\in\{1,\ldots, I\}$  and apply Proposition \ref{propHN} to the function $u_k$ restricted to $B_\delta(x^{(i)})$ together with the $N=L_i\ge 1$ blow-up sequences converging to $x^{(i)}$, hence getting
$$\lim_{k\to\infty}\int_{B_\delta(x^{(i)})}V_ke^{2mu_k}dx= L_i\Lambda_1.$$
Moreover, since $u_k\to-\infty$ uniformly locally in $\Omega\setminus S$, it follows that
$$\lim_{k\to\infty}\int_{\Omega_0\setminus\bigcup_{i=1}^I B_\delta(x^{(i)})}V_ke^{2mu_k}dx=0,$$
whence \eqref{quant} and \eqref{quant2} follow at once.
\hfill$\square$

\section{Proof of Theorem \ref{trm2}}\label{sec2}

Here the Harnack-type estimates of \cite{rob2} are replaced by a technique of \cite{DR}, reminiscent of the Pohozaev inequality. For this it is crucial to have the gradient estimates of Propositions \ref{prop4bis} and \ref{gradbis} below, which correspond to (and in fact are stronger than) Propositions \ref{prop4} and \ref{grad} of the previous section, and which also work in the case $m=1$.
 
\begin{prop}\label{prop3bis} Let $(u_k)$ be a sequence of solutions to \eqref{eq0}, \eqref{area} and \eqref{Vkbis} satisfying \eqref{nablau2} for some ball $B_\rho(\xi)\subset\Omega$, and let $S$ be as in \eqref{defS}.
Then $S$ is finite (possibly empty) and one of the following is true:
\begin{itemize}
\item[(i)] $u_k\to u_0$ in $C^{2m-1}_{\loc}(\Omega\backslash S)$ for some $u_0\in C^{2m}(\Omega\backslash S)$;
\item[(ii)] $u_k\to -\infty$ locally uniformly in $\Omega \backslash S$.
\end{itemize}
If $S\neq \emptyset$ and $V_0(x)>0$ for some $x\in S$, then case (ii) occurs.
\end{prop}

\begin{proof}
The proof is analogous to the proof of Proposition \ref{prop3}. Following that proof and its notation, it is enough to show that if case $(b)$ occurs, then $\Gamma=\emptyset$. In order to show this, observe that $\nabla\varphi\equiv 0$ in $\Omega\backslash S$. Otherwise, since $\nabla \varphi$ is analytic, we would have
$$\int_{B_\rho(\xi)}|\nabla\varphi|dx>0,$$
where $B_\rho(\xi)\subset\Omega$ is as in \eqref{nablau2}. Then \eqref{ubeta} would imply
$$\lim_{k\to\infty}\int_{B_\rho(\xi)}|\nabla u_k|dx=\infty,$$
contradicting \eqref{nablau2}. Therefore $\varphi\equiv const$ and \eqref{phi} implies that $\varphi<0$ in $\Omega\backslash S$, i.e. $\Gamma=\emptyset$, as claimed.
\end{proof}

This completes the proof of the first part of Theorem \ref{trm2} and, as we did in the last section, we shall now assume that $(u_k)$ satisfies all the hypothesis of Theorem \ref{trm2}, including \eqref{nablau3}. As before, if $S=\emptyset$ the proof of Theorem \ref{trm2} is complete, hence we shall also assume that $S\neq \emptyset$ and we shall prove that we are in case $(ii)$ of the theorem.

\begin{prop}\label{prop4bis}
For every open set $\Omega_0\subset\subset \Omega\backslash S$ there is a constant $C=C(\Omega_0)$ such that
\begin{equation}\label{stimebis}
\|u_k-\bar u_k\|_{C^{2m-1}(\Omega_0)}\leq C,
\end{equation}
where $\bar u_k:=\intm_{\Omega_0}u_kdx$.
\end{prop}

\begin{proof}
If case $(i)$ of Proposition \ref{prop3bis} occurs the proof is trivial, hence we shall assume that we are in case $(ii)$.
Consider an open set $\tilde\Omega_0\subset\subset \Omega\setminus S$ with smooth boundary and with $\Omega_0\subset\subset\tilde \Omega_0$. 
Write $u_k=w_k+h_k$ in $\tilde \Omega_0$, with $\Delta^m h_k=0$ and $w_k=\Delta w_k=\ldots=\Delta^{m-1}w_k=0$ on $\partial \tilde \Omega_0$. Since
$$|\Delta^m w_k|=|\Delta^m u_k|\le C=C(\tilde\Omega_0)\quad \text{on }\tilde\Omega_0,$$
by elliptic estimates we have
$$\|w_k\|_{C^{2m-1}(\tilde \Omega_0)}\leq C.$$
This and \eqref{nablau3} give $\|\nabla h_k\|_{L^1(\tilde \Omega_0)}\leq C$, hence, since $\Delta^m(\nabla h_k)=0$, by elliptic estimates we infer
$$\|\nabla h_k\|_{C^\ell(\Omega_0)}\le C=C(\ell,\Omega_0,\tilde\Omega_0)$$
for every $\ell\ge 0$, see e.g. Proposition 4 in \cite{mar1}.
Therefore
$$\|\nabla u_k\|_{C^{2m-2}(\Omega_0)}\le C=C(\Omega_0,\tilde\Omega_0),$$
and \eqref{stimebis} follows at once.
\end{proof}

\begin{prop}\label{gradbis} For every open set $\Omega_0\subset\subset \Omega$ there is a constant $C$ independent of $k$ such that 
\begin{equation}\label{grad2bis}
\int_{B_r(x_0)}|\nabla^{\ell} u_k|dx\leq C r^{2m-\ell},
\end{equation}
for $1\leq \ell \leq 2m-1$ and for every ball $B_r(x_0)\subset \Omega_0$.
\end{prop}

\begin{proof} Going back to the proof of Proposition \ref{grad}, we only need to replace \eqref{poisson} by
\begin{equation}\label{poissonbis}
\begin{split}
u_k(x)-\bar u_k=&\int_{B_{4\delta}(\xi)} G_x(y)\Delta^{m}u_k(y)dy\\
&+\sum_{j=0}^{m-1}\int_{\de B_{4\delta}(\xi)}\frac{\de}{\de\nu}(\Delta^{m-j-1}G_x) \Delta^j (u_k-\bar u_k)d\sigma,
\end{split}
\end{equation}
where now
$$\Delta^{m}G_x=\delta_x \text{ in }B_{4\delta}(\xi),\quad G_x =\Delta G_x=\ldots=\Delta^{m-1}G_x=0\text{ on }\de B_{4\delta}(\xi),$$ 
and
$$\bar u_k:=\Intm_{B_{4\delta}(\xi)}u_kdx.$$
Differentiating and using $|\nabla^{\ell}G_x(y)|\leq \frac{C}{|x-y|^\ell}$ (see e.g. \cite{DAS}) and \eqref{stimebis} (with $\Omega_0=B_{4\delta}(\xi)$) on $\de B_{4\delta}(\xi)$, we infer for $x\in B_{2\delta}(\xi)$
\begin{equation*}
|\nabla^{\ell} u_k(x)|\leq C\int_{B_{4\delta}(\xi)}\frac{e^{2mu_k(y)}}{|x-y|^\ell}dy+C.
\end{equation*}
Integrating on $B_r(x_0)\subset B_{2\delta}(\xi)$ and using Fubini's theorem as before, we finally get
\begin{equation*}
\int_{B_r(x_0)}|\nabla^{\ell} u_k(x)|dx\leq C\int_{B_r(x_0)}\int_{B_{4\delta}(\xi)}\frac{e^{2mu_k(y)}}{|x-y|^\ell}dydx +Cr^{2m}\le Cr^{2m-\ell}.
\end{equation*}
\end{proof}

Proposition \ref{exh} also holds with the same proof. Proposition \ref{propgrad2} has the following analogue, which can be proved as above. Notice that at this point we are not yet excluding that $L>I$.

\begin{prop}\label{propgrad2bis} For $1\leq \ell\leq 2m-2$ and $\Omega_0\subset\subset\Omega$ we have
\begin{equation}\label{nablaubis}
\inf_{1\leq i\leq L}|x-x_{i,k}|^\ell |\nabla^\ell u_k(x)|\leq C=C(\Omega_0), \quad \text{for }x\in\Omega_0.
\end{equation}
\end{prop}

Taking into account Proposition \ref{exh} and Proposition \ref{propgrad2bis}, one can follow the proof of step 4 of Theorem 2 in \cite{mar3}, in order to prove that the concentration points are isolated, i.e. $x^{(i)}\neq x^{(j)}$ for $i\neq j$, $I=L$, and that for $\delta>0$ small enough
$$\lim_{R\to \infty}\lim_{k\to\infty}\int_{B_\delta(x_{i,k})\setminus B_{R\mu_{i,k}}(x_{i,k})}V_k e^{2mu_k}dx=0.$$
This and \eqref{energia} complete the proof of Theorem \ref{trm2}.

\section{A few open questions}

\noindent\emph{1) Necessity of hypothesis \eqref{Deltau-} and \eqref{nablau3}.} Is the assumption \eqref{Deltau-} (resp. \eqref{nablau3}) necessary in order to have quantization in the second part of Theorem \ref{trm1} (resp. Theorem \ref{trm2}), or is \eqref{Deltau} (resp. \eqref{nablau2}) enough?

For instance, is it possible to find a sequence $(u_k)$ of solutions to
$$(-\Delta)^m u_k=e^{2mu_k}\quad \text{in }B_1(0) $$
with 
$$\lim_{k\to\infty}\int_{B_1(0)}e^{2mu_k}dx= \alpha\in(0,\Lambda_1) $$
and
$$\int_{B_\rho(\xi)}|\Delta u_k|dx\leq C$$
for a ball $B_\rho(\xi)\subset B_1(0)$? To our knowledge, this is unknown even in the case when $u_k$ is radially symmetric, see \cite{rob1}.

\medskip

\noindent\emph{2)} If case (i) of Theorem \ref{trm1} (or equivalently Theorem \ref{trm2}) occurs, is it possible to have $S\neq\emptyset$? If instead of \eqref{Vk} we only assume the bound $\|V_k\|_{L^\infty(\Omega)}\leq C$, the answer is negative, as shown for $m=1$ by Shixiao Wang \cite{wan}.

\medskip

\noindent\emph{3) Boundedness from above.} Given a solution $u$ to
$$(-\Delta)^m u= Ve^{2mu}\quad\text{in }\R{2m},$$
with $V\in L^\infty(\R{2m})$, $e^{2mu}\in L^1(\R{2m})$, is it true that $\sup_{\R{2m}}u<\infty$?

For $m=1$ this was proven by Br\'ezis and Merle, \cite[Theorem 2]{BM}, but their simple technique, which rests on the mean-value theorem for harmonic functions, cannot be applied when $m>1$. It is only known that when $V\equiv const\ge 0$ the answer is positive, see \cite[Theorem 1]{lin}, \cite[Theorem 1]{mar1} and \cite[Theorem 3]{mar2}.

\end{document}